\documentclass[12pt]{article}

\usepackage[british]{babel}
\usepackage{mathrsfs}
\usepackage{amsmath}
\usepackage{amssymb}
\usepackage{amsthm}
\usepackage{graphics}
\usepackage{enumerate}
\usepackage{epic}
\usepackage{color}

\newtheorem{lemma}{Lemma}[section]

\newtheorem{theorem}[lemma]{Theorem}
\newtheorem{definition}{Definition}
\newtheorem{remark}{Remark}
\newtheorem{conj}{Conjecture}
\newcommand{\FF}{\mathbb F}
\newcommand{\PG}{\mathrm{PG}}
\newcommand{\sE}{\mathscr{E}}
\newcommand{\sC}{\mathscr{C}}
\newcommand{\sD}{\mathscr{D}}
\newcommand{\sV}{\mathscr{V}}

\newcommand{\cB}{\mathcal{B}}
\newcommand{\sF}{\mathscr{F}}
\newcommand{\sS}{\mathscr{S}}

\title{Near-MDS codes from elliptic curves}
\author{Angela Aguglia\footnote{Angela Aguglia: angela.aguglia@poliba.it\hfill\newline\hspace*{1.4em}
\hspace{-1.77em}Dipartimento di Meccanica, Matematica e Management - Politecnico di Bari - Via Orabona, 4 - 70126 Bari (Italy).}
\and
Luca Giuzzi\footnote{Luca Giuzzi: luca.giuzzi@unibs.it\hfill\newline\hspace*{1.4em}
\hspace{-1.77em}DICATAM - University of Brescia - Via Branze 53 -
  I-25123 Brescia, (Italy)}
\and
Angelo Sonnino\footnote{Angelo Sonnino: angelo.sonnino@unibas.it\hfill\newline\hspace*{1.4em}
\hspace{-1.77em}Dipartimento di Matematica, Informatica ed Economia - Universit\`{a} degli Studi della Basilicata - Viale dell'Ateneo Lucano, 10 - 85100 Potenza (Italy).}
}

\date{}

\begin{document}

\maketitle

\begin{abstract}
We provide a geometric construction of $[n,9,n-9]_q$ near-MDS codes arising from elliptic curves with $n$ $\FF_q$-rational points. Furthermore, we show that in some cases  these codes cannot be extended to  longer near-MDS codes.

\vspace{\baselineskip}\noindent\textbf{Key words:} linear code; near-MDS code; elliptic curve.

\vspace{\baselineskip}\noindent\textbf{Mathematics subject classification:} 94B05, 51A05, 51E21.
\end{abstract}

\section{Introduction}

Maximum distance separable (for short MDS) codes are the best linear $[n,k,d]_{q}$ codes as they meet the Singleton bound, that is, $n=d+k-1$. The non-negative integer $s({\bf C}):= n-k+1-d$ is said to be the Singleton defect of the code ${\bf C}$. Thus, the Singleton defect of an MDS code is zero.

A linear code $\bf C$ is defined to be a near-MDS (for short NMDS) code if $s({\bf C})=s({\bf C}^{\bot})=1$  where ${\bf C}^{\bot}$ is the dual code of ${\bf C}$. Hence, a NMDS $[n,k]$ code has minimum distance $n-k$.

NMDS codes were introduced by Dodunekov and Landjev \cite{dodunekov1995} with the aim of constructing good linear codes by slightly weakening the restrictions in the definition of an MDS code. NMDS codes have similar properties to MDS codes. Some non-binary linear codes such as the ternary Golay codes, the quaternary quadratic residue $[11,6,5]_{4}$-code, and the quaternary extended quadratic residue  $[12,6,6]_{4}$-code are notable examples of NMDS codes; see~\cite{MS77}.

The geometrical counterpart of an  NMDS code is an  $n$-track in a Galois space which is a set of $n$ points in an $N$-dimensional Galois space such that every $N$ of them are linearly independent but some $N+1$ of them, see \cite{BF}.
If every $N+2$ points of the $n$-track generate the whole space then the $n\times (N+1)$ matrix whose columns are homogeneous coordinates of  the $n$-track points is a generator matrix of an NMDS code. The $n$-track is complete, i.e. maximal with respect to set theoretical inclusion, if and only if the code is not extendable.

Let $N_q$ denote  the  maximum  number  of $\FF_q$-rational points on an elliptic curve defined over $\FF_q$; it is well-known that,
by Hasse theorem,
$|N_q-(q+1)|\leq 2\sqrt{q}$.

NMDS codes of length up to $N_q$ may be constructed from elliptic curves. An interesting question is whether there exist NMDS codes of length greater than $N_q$. Constructions of NMDS codes from elliptic curves are found in \cite{abatangelo2005,abatangelo2008,giulietti} where results both from combinatorics and algebraic geometry are used.

Here we provide a  geometric construction of $9$ dimensional NMDS codes using an algebraic curve of order $9$ in $\PG(9, q)$ which arises from a non-singular cubic curve  $\sE: f(X,Y,Z)=0$  of $\PG(2, q)$ via the (modified) Veronese embedding:
\begin{multline}
  \label{ver}
	\nu_{3}^2 :(X{:}Y{:}Z)\mapsto \\
	\big(f(X,Y,Z)\,{:}\,X^2Y\,{:}\,X^2Z\,{:}\,XY^2\,{:}\,XYZ\,{:}\,XZ^{2}\,{:}Y^3\,{:}\,Y^2Z\,{:}\,YZ^{2}\,{:}\,Z^{3}\big).
\end{multline}

We also show that certain  codes from elliptic curves are not extendible to longer NMDS codes. The proof depends on some results on the number of $\FF_q$-rational lines through a given point $P$ that meet a plane elliptic curve in exactly three $\FF_q$-rational points and on some computations carried out with the aid of GAP~\cite{gap}.

\section{Preliminaries}

The following definitions of an NMDS code of length $n$ and dimension $k$ over a finite field $\mathbb{F}_{q}$ are equivalent to that given in the Introduction; see \cite{dodunekov2000}.
\begin{definition}
	\label{def:generator}
	{A linear $[n,k]$ code over $\FF_{q}$ is NMDS if any of its generator matrices, say $G$, satisfies the following conditions:
	\begin{enumerate}[(i)]
		\item any $k-1$ columns of $G$ are linearly independent;
		\item $G$ contains $k$ linearly dependent columns;
		\item any $k+1$ columns of $G$ have full rank.
	\end{enumerate}}
\end{definition}
\begin{definition}
	\label{def:paritycheck}
	{A linear $[n,k]$ code over $\FF_{q}$ is NMDS  if any of its parity check matrices, say $H$, satisfies the following conditions:
	\begin{enumerate}[(i)]
		\item any $n-k-1$ columns of $H$ are linearly independent;
		\item $H$ contains $n-k$ linearly dependent columns;
		\item any $n-k+1$ columns of $H$ have full rank.
	\end{enumerate}}
\end{definition}

From a geometric point of view, a NMDS $[n,k]$ code $\mathbf{C}$ over $\FF_{q}$ can be regarded as a projective system (i.e. a distinguished point set)  $\mathbf{C}$ in a projective space $\PG(k-1,q)$; see~\cite{TVZ} for more details.

\begin{definition}
	\label{def:points}
	{A subset $\mathbf{C} \subseteq \PG(k-1,q)$ is an $(n; k, k-2)$-set in $\PG(k-1,\mathbb{F}_{q})$ if  it satisfies the following conditions:
	\begin{enumerate}[(i)]
		\item\label{land1} every $k-1$ points in $\mathbf{C}$ span a hyperplane of $\PG(k-1,q)$;
		\item\label{land2} there exists a hyperplane of $\PG(k-1,q)$ containing exactly $k$ points of $\mathbf{C}$;
		\item\label{land3} every $k+1$ points of $\mathbf{C}$ generate the whole $\PG(k-1,q)$.
	\end{enumerate}}
    \end{definition}

\begin{definition}
	\label{def:complete}
{An $(n; k, k-2)$-set in $\PG(k-1,\mathbb{F}_{q})$ is complete if it is maximal with respect to set-theoretical inclusion.}
\end{definition}
Thus, in this setting, an NMDS $[n,k]$ code  over $\FF_{q}$ is an $(n; k, k-2)$-set in $\PG(k-1,\mathbb{F}_{q})$.

Given an integer $\nu\geq 1$ and a prime power $q=p^{h}$, consider the set $\mathfrak{C}^{\nu}$ of all the curves of degree $\nu$ contained in the projective plane $\PG(2,q)$ over a finite field $\mathbb{F}_{q}$. Since any curve $\sC\in\mathfrak{C}^{\nu}$ is uniquely determined by $m+1=\binom{\nu +2}{2}$ parameters in $\mathbb{F}_{q}$, that is, the coefficients of its equation
\begin{multline*}
	a_{0}Z^{\nu}+(a_{1}X+a_{2}Y)Z^{\nu -1}+(a_{3}X^{2}+a_{4}XY+a_{5}Y^{2})Z^{\nu -2}+\cdots \\
	+(a_{m-\nu }X^{\nu}+a_{m-\nu +1}X^{\nu -1}Y+\cdots +a_{m-1}XY^{\nu -1}+a_{m}Y^{\nu})=0,
\end{multline*}
and the curve is unchanged if these parameters are multiplied by a common factor, then $\mathfrak{C}^{\nu}$ can be regarded as a projective space $\PG(m,q)$ with homogeneous coordinates $(a_{0}{:}a_{1}{:}\cdots{:}a_{m})$. We may also denote a curve $\sC$ by using its defining polynomial.


The following result---which is an implicit formulation of the famous Cayley-Bacharach theorem---will be useful later; see \cite{eisenbud1996}.
\begin{theorem}
  \label{teo:bacharach}
  Let $\sE$ and $\sC$
  be two distinct cubic curves meeting in a set
  $\sS$ consisting of $9$ points (counted with multiplicities).
  If $\sD\subset\PG(2,q)$ is any cubic curve containing all but
  one point of $\sS$, then $\sC\cap\sD=\sS$.
\end{theorem}

\section{Lifting point sets}
\label{sec:lift}

The space $\mathfrak{C}^{3}$ consisting of all the cubics in $\PG(2,q)$ has projective dimension $9$, hence $10$ independent cubic curves are required to generate it. Let $\sE$ be a non-singular cubic curve  of equation $ f(X,Y,Z)=0$  over $\FF_q$.
A suitable basis $\cB$ for $\mathfrak{C}^{3}$, containing $\sE$, can be written by using the following polynomials:
$$\cB=\{f(X,Y,Z),\ X^{2}Y ,\ X^2Z,\ XY^2,\ XYZ,\ XZ^2,\ Y^3,\ Y^{2}Z,\ YZ^{2},\ Z^{3}\},$$ where $f(X,Y,Z)$ is required  to contain the term $X^{3}$.
 In fact, the defining polynomial of any cubic curve would be suitable as first element of the basis $\cB$, as long as it contains the monomial $X^3$; nevertheless, the choice of an elliptic curve is motivated by the fact that,
unlike the case of genus $0$, the number of $\FF_q$-rational points of a carefully chosen elliptic curve is not necessarily limited to $q+1$.

We consider the following embedding of the points of $\PG(2,q)$ onto $\PG(9,q)$ with projective coordinates $(X_{0}{:}X_{1}{:}X_{2}{:}X_{3}{:}X_{4}{:}X_{5}{:}X_{6}{:}X_{7}{:}X_{8}{:}X_{9})$
by means of the mapping
$\nu_3^2:\PG(2,q)\to
\PG(9,q)$ ~\eqref{ver} which is a Veronese embedding of degree $3$.
Let $\sV_3$ be the image of
$\nu_3^2$; clearly $\sV_3$ is (projective equivalent to) the cubic Veronese surface.

      More in detail, the points of the curve $\sE$ are mapped onto a curve $\Gamma$ of $\PG(9,q)$ with the same number $n$ of $\FF_q$-rational points as $\sE$.
      Also $\Gamma$ is the complete intersection of $\sV_3$ with the hyperplane $\Sigma\cong\PG(8,q)$ of equation $X_0=0$.
      Since  for every
      cubic curve $\sC$ of equation $g(X,Y,Z)=0$ in $\PG(2,q)$, the defining polynomial is a linear combination of the elements of $\cB$, that is,
\begin{multline*}
	g(X,Y,Z)=\lambda_{0}f(X,Y,Z)+\lambda_{1}Y^{3}+\lambda_{2}XZ^{2}+\lambda_{3}YZ^{2}+\lambda_{4}X^{2}Z+ \\
	\lambda_{5}Y^{2}Z+\lambda_{6}XYZ+\lambda_{7}X^{2}Y+\lambda_{8}XY^{2}+\lambda_{9}Z^{3},
\end{multline*}
it turns out that $\nu_3^2(\sC)$ is the
complete intersection of $\sV_3$ with the
hyperplane $\Pi\subset\PG(9,q)$ of equation
\begin{equation}
	\label{eq:hyperplane}
	\sum_{i=0}^{9}\lambda_{i}X_{i}=0,
\end{equation}
which is distinct from $\Sigma$.
 Thus, every cubic curve  $\sC: g(X,Y,Z)=0$ of $\PG(2, q)$  corresponds to a hyperplane of equation \eqref{eq:hyperplane}.
 Back to $\PG(2,q)$, the set $(\nu_3^2)^{-1}(\Pi\cap\sV_3)$ corresponds to a unique cubic curve $\sC$ distinct from $\sE$, and,
 clearly, $(\nu_3^2)^{-1}(\Pi\cap\Gamma )$ corresponds to $\sC\cap\sE$.

\begin{theorem} \label{AGS}Suppose that $\sE$ has $n \geq 9$ points. Then the point set $\Gamma$ is an $(n; 9,7)$-set in $\Sigma =\PG(8,q)$.
\end{theorem}
\begin{proof}
To prove the theorem it suffices to consider the mutual position of cubic curves in $\PG(2,q)$.
	\begin{enumerate}[(i)]
		\item Take eight distinct points $P_{1},\ldots,P_{8}\in\Gamma$ and consider the corresponding distinct  points $Q_{1},\ldots, Q_{8}\in  \sE$, with $Q_{i}=(\nu_3^2)^{-1}(P_{i})$. Suppose that there is a $t$-dimensional net with $t\geq 2$, say $\sF$, consisting of cubics through $Q_{1},\ldots ,Q_{8}$. Then, from Theorem \ref{teo:bacharach} there is a ninth point $Q_{9}\in\sE$ such that the points $Q_{1},\ldots ,Q_{9}$ are in the support of $\sF$. This implies that every further point $Q_{10}\in\sE\setminus\{Q_{1},\ldots ,Q_{9}\}$ yields a $(t-1)$-dimensional net consisting of cubics through $Q_{1},\ldots ,Q_{9}$ which are distinct from $\sE$ and have ten points in common with it, contradicting B\'{e}zout's theorem. Hence, $\sF$ must be a pencil of cubic curves in $\PG(2,q)$ including $\sE$ and passing through $Q_{1},\ldots ,Q_{8}$. Back to $\PG(9,q)$, we observe that $\sF$ corresponds to a pencil of hyperplanes of $\PG(9,q)$ which meet in a unique $7$-dimensional subspace $\Delta$ such that $\{P_{1},\ldots ,P_{8}\}\subset (\Gamma\cap\Delta )$, that is, $P_{1}$, \textellipsis , $P_{8}$, generate the hyperplane  $\Delta$ of $\Sigma$.
		\item  From Theorem \ref{teo:bacharach}, there is a further point $Q_{9}\in\PG(2,q)$ which belongs to the intersection of $\sE$ and all the other cubics of the above pencil $\sF$. This proves that the  previous  subspace $\Delta$ meets $\Gamma$ in $P_{1}$,\textellipsis ,$P_{8}$, $P_{9}=\nu_3^2 (Q_{9})$.
		\item  Let $\Pi$ be a hyperplane of $\PG(9,q)$ different from $\Sigma$. Put $\sC=(\nu_3^2)^{-1}(\Pi )$. From B\'{e}zout's theorem we know that $\vert\sE\cap\sC\vert\leq 9$, therefore any hyperplane of $\PG(9,q)$ has at most $9$ points in common with $\Gamma$. Hence, $\Gamma$ is a curve of order $9$, therefore $10$ points of $\Gamma$ generate the whole $\Sigma$.
	\end{enumerate}
	The claim follows.
\end{proof}
\begin{remark} The code associated to $\Gamma$ can also be interpreted as an AG-code, see \cite{TVZ}. Indeed, Theorem \ref{AGS} is a consequence of  \cite[Theorem 4.4.19]{TVZ}. However, our  proof does not use the Riemmman-Roch Theorem.
\end{remark}
\section{Some complete NMDS codes}

In this section we provide some examples of complete NMDS codes in the set of codes constructed above by lifting the elliptic curve  $\sE$ in the case when the base field is large enough.

By Definition~\ref{def:complete}, the algebraic curve $\Gamma=\nu_3^2 (\sE)$ provides a complete NMDS code, that is a complete $(n;9,7)$-set of $\PG(8,q)$, if and only if for any $Q\in\Sigma$ there exists at least one hyperplane $\Pi$ of $\Sigma$ with $Q\in\Pi$ meeting $\Gamma$ in $9$ points.

\begin{definition}
We call a point $Q\in\Sigma$ \emph{special} for $\Gamma$  if for all hyperplanes $\Pi$ of $\Sigma$ through $Q$ we have $|\Pi\cap\Gamma|<9$.
\end{definition}


We expect that for large $q$ special points, if they exist at all, are very few;
see Lemma~\ref{specialp}.
So we propose the following conjecture.

\begin{conj}
  \label{conj1}
  Suppose $q\geq 121$ to be such that $2,3\not| q$. Then there are no special points for $\Gamma$.
\end{conj}

In order to verify Conjecture~\ref{conj1}, we performed some computer
searches for some values of $q$. For $q \in\{7,11,13\}$ we executed a (non-trivial) exhaustive
search. For $q\geq 121$ we provide an argument showing that there cannot be too many special points,
if they exist at all.
We leave the solution of the problem and its generalization to a future work.

\subsection{Search for small $q$}
Recall that any $8$ distinct points of $\sV_3$ are linearly independent; see~\cite{KS13}.

For small values of $q$ it is possible to perform an exhaustive
search, adopting the following procedure:
\begin{enumerate}
\item Let $\Gamma=\nu_3^2 (\sE)$ be the embedding of $\sE$;
\item for any set of $9$
  points of $\Gamma$, consider the matrix containing
  their components; let $\mathfrak{G}$ be the list of such matrices
  having rank $8$. In particular, each
  element of $\mathfrak{G}$ corresponds to a hyperplane meeting
  $\Gamma$ in $9$ points. We call such hyperplanes \emph{good}.
\item For each matrix $H\in{\mathfrak{G}}$, let $H'$ be
  a column vector spanning the kernel of $H$. In particular,
  we have that a row vector $v$ belongs to the span of the rows of
  $H$ if and only if $vH'=\mathbf{0}$.
\item Consider the linear code $C$ with parameters $[|{\mathfrak{G}}|,9]$
  whose generator matrix $G$ consists of all columns of the form $H'$ as
  $H$ varies in $\mathfrak{G}$. A point $P$ represented by a vector
  $v$ can be added to $\Gamma$ if, and only if, $P$ does not
  belong to any of the hyperplanes represented by the columns of $G$;
  in other words $P$ can be added to $\Gamma$ if and only if the
  word $PG$ corresponding to $P$ does not contain any $0$-component.
\end{enumerate}
Using the above argument, we can state the following.
\begin{theorem}
  The (n;9,7)-set $\Gamma$ is complete if and only if the code $C$ with
  generator matrix $G$ constructed above does not contain any
  word of maximum weight $n$.
\end{theorem}
Clearly, it is not restrictive to replace the code $C$ with
a code $C'$ equivalent to $C$. In particular, if we transform its generator matrix
$G$ to row-reduced echelon form, we see that no point with
at least a $0$ component can give a word of $C'$ of weight $n$;
this allows to exclude from the search all points whose
transforms (under the operations yielding the reduction of $C$)
lie on the coordinate hyperplanes.

We now limit ourselves to the odd order case with $q$ not divisible by $3$.
Then any elliptic curve $\sE$ of $\PG(2,q)$ admits an equation in
canonical Weierstrass form
\[ Y^2 = X^3+aX+b, \]
with $a,b\in \FF_q$ such that $-16(4a^3+27b^2)\neq 0$;
see~\cite{S86}.

\begin{remark}
  Good hyperplanes correspond to linear systems of cubic curves
  cutting $\sE$ in $9$ points; by~\cite[Theorem 43]{KM20}, we
  see that the number of such hyperplanes is approximately
  $\frac{1}{9!}q^7$.

  We leave to a future work to determine exactly what sets of $9$ distinct
  points of a given elliptic curve $\sE$ might arise as intersection
  divisor with another curve, in other terms to determine what the good
  hyperplanes are.

  Our Conjecture~\ref{conj1} can be restated by saying that
  the union of all good hyperplanes for $\sE$ is $\PG(8,q)$ for $q$
  sufficiently large.
\end{remark}

We can now apply
the aforementioned strategy for all possible
values of $a,b$ yielding elliptic curves. This leads to the following.
\begin{theorem}
  Suppose $q\in\{7,11,13\}$. Then, the lifted $(n;9,7)$-set $\Gamma$ in $\PG(8,q)$
is complete if and only if $n=|\sE|\geq 15$. In particular, for
$q=7$ the lifted set $\Gamma$ is never complete.
\end{theorem}

\subsection{Properties for large $q$}

We now provide an argument to prove that there might not be too many special points. This makes it possible to verify for several values of $q$ that the $(n; 9,7)$-set $\Gamma$ in $\Sigma=\PG(8,q)$ is complete and gives evidence supporting Conjecture~\ref{conj1}.

As in the previous section, the projective plane $PG(2,q)$ is assumed to be of order $q$ odd and
not divisible by $3$. Furthermore we suppose $q\geq 121$. Let $j(\sE)$ be the $j$-invariant of $\sE$, that is the six cross-ratios of the four tangents from a point of $\sE$ to other points of $\sE$. We limit ourselves to the case $j(\sE)\neq 0$, see \cite[Theorem 11.15]{H}.

 We will use the following result which is a direct consequence of \cite[Lemma 3.2]{giulietti}.

 \begin{lemma}\label{trisecants} Let $q\geq 121$ and consider
   an elliptic cubic $\sE(\FF_q)$ with $j(\sE)\neq0$. Then there are
   at least $7$ trisecant $\FF_q$-rational
   lines through any given $\FF_q$-rational point.
 \end{lemma}
Up to a change of projective reference,
we can assume without loss of generality that the curve $\sE$ in $\PG(2,q)$ is met by the reducible cubic $XYZ=0$ in $9$ distinct
$\FF_q$-rational points.

\begin{lemma}
  \label{specialp}
  Under the assumption $q\geq 121$ any special point $Q\in\Sigma$
  has to be a point $Q=(0,q_1,q_2,\dots,q_9)\in \Sigma \setminus \Gamma$  such that $[q_1,q_3,q_4],[q_4,q_7,q_8] \in \sE$ and one of the following conditions holds
  \begin{itemize}
  \item $q_1,q_7=0$; $q_3,q_4,q_8\neq 0$;
  \item $q_1,q_8=0$; $q_3,q_4,q_7\neq 0$;
  \item $q_3,q_7=0$; $q_1,q_4,q_8\neq 0$;
  \item $q_3,q_8=0$; $q_1,q_4,q_7\neq 0$.
  \end{itemize}
  \end{lemma}
\begin{proof}
 Let $Q=(0,q_1,q_2,\dots,q_9)\in\Sigma$.
 If $Q\in\Gamma$, then $Q$ is not special; indeed,
 if $Q\in\Gamma$, then $Q=\nu_3^2(P)$ with $P\in\sE$.
  Consider a reducible cubic curve $\sC$
  in $\PG(2,q)$, union of $3$ lines
  $\ell,m,r$ with $P \in\ell \setminus \{m \cup r\} $ and such that
  $|(\ell\cup m\cup r)\cap\sE|=9$. Such a curve if $|\sE|>9$
  is guaranteed to exist by Lemma~\ref{trisecants}
  and it corresponds to
  a hyperplane of $\PG(9,q)$ through $Q$
  meeting $\Gamma$ in $9$ distinct points.
  So $Q$ is not special.

  Now consider a cubic curve $\sC$
  in $\PG(2,q)$ with equation of the form
  \begin{equation}\label{eq1}
   YZ(\alpha X+\beta Y+\gamma Z)=0,
   \end{equation}
   and a cubic curve  $\sC'$ with equation of type
   \begin{equation}\label{eq2}
     XY(aX+bY+cZ)=0.
     \end{equation}
  Via the Veronese embedding $\nu_3^2$, $\sC$ corresponds to the hyperplane of
  equation $\alpha X_4+\beta X_7+\gamma X_8=0$, whereas $\sC'$
  corresponds to the hyperplane
  $aX_1+bX_3+cX_4=0$.

  For any $Q\in\Sigma\setminus\Gamma$ write $P_Q:=[q_4,q_7,q_8]$ and $P'_Q:=[q_1,q_3,q_4]\in\PG(2,q)$.

    If $P_Q\not\in{\sE}$, by Lemma~\ref{trisecants} there are at least
    $7$ lines through $P_Q$ meeting ${\sE}$ in $3$ distinct
    points; in particular there is at least one line of equation
    $\alpha X+\beta Y+\gamma Z=0$
    through $P_Q$
    meeting $\sE\setminus([Y=0]\cup [Z=0])$ in $3$
    distinct points.
    Consequently  the cubic
    $\sC:YZ(\alpha X+\beta Y+\gamma Z)=0$ corresponds to a hyperplane $\Pi$ of $\PG(9,q)$ through $Q$,
    meeting $\Gamma$ in $9$ distinct points and we are done.

  If $P_Q \in {\sE}$ but $P'_Q\not\in\sE$, repeating the same argument starting from a  cubic $\sC'$ with equation~\eqref{eq2}, we see that
  $Q$ is  not special.

    Thus, we suppose  $P_Q, P'_Q \in\sE$ and  distinguish several
    cases:
    \begin{enumerate}
    \item If $q_4=0$, then the cubic
      $\sC$ of
      equation $XYZ=0$ corresponds to the hyperplane $X_4=0$ passing through
      $Q$ with $9$ intersections with $\Gamma$.
    \item If $q_4\neq 0$ and $q_7=q_8=0$, then $P_Q=[1,0,0]\not\in
      {\sE}$, which is excluded.

      \item If $q_4\neq 0$ and $q_1=q_3=0$, then $P'_Q=[0,0,1]\not\in
      {\sE}$, which is excluded.
    \item\label{a4} Let $q_4\neq 0$ with  $q_7 \neq 0$ and $q_8 \neq 0$, then $P_Q$ is not on
      $[Y=0]\cup [Z=0]$ in $\PG(2,q)$.
      Then, from Lemma \ref{trisecants} there are at least $7$ lines in $\PG(2,q)$ through
      $P_Q$ which are $3$-secants to $\sE$.
    Since $\sE$ has $6$ points on the union of the lines
    $[Y=0]$ and $[Z=0]$, there is at least one line through $P_Q$
    with equation: $\alpha_1 X+\beta_1 Y+\gamma_1 Z=0$  meeting
    $\sE$ in $3$ points none of which is on $[Y=0]$ and $[Z=0]$.
    So, the  hyperplane of $\PG(9,q)$ through $Q$,  corresponding to the cubic $\sC: YZ(\alpha_1 X+\beta_1 Y+\gamma_1 Z)=0$  meets $\Gamma$ in $9$
    points.
  \item Let $q_4\neq 0$ , $q_7\neq 0$ and $q_8=0$ (or, equivalently, $q_4\neq 0$, $q_7=0$ and
    $ q_8\neq 0$).
    Using  an  argument similar to that of point~\ref{a4}.\ but starting from a  cubic $\sC'$  through $P'_Q$ with equation of the form~\eqref{eq2},
    it turns out that  if  $q_1\neq 0$ and $q_3\neq 0$  then the points $Q(0,q_1,q_2, \ldots, q_7,0,q_9)$
    (or $Q(0,q_1,\dots,q_6,0,q_8,q_9)$)
    are not special.
\end{enumerate}
Thus, our lemma follows.
\end{proof}

\begin{remark}
  \label{m:rem}
  Let $Q=(0,q_1,\dots,q_9)\in\Sigma$ such that $Q$ is not ruled out as special point in
  Lemma~\ref{specialp}. For instance,
  suppose $q_8=0$ and either $q_1=0$ or $q_3=0$ with $[q_1,q_3,q_4]\in\sE$.
  So, take $P(a,0,1)\in\PG(2,q)\setminus\sE$ and consider a cubic $\sC$
     with equation: $Y(Y-m_1X+a m_1Z)(Y-m_2 X+a m_2Z)=0$
     passing through $P$ meeting $\sE$ in $9$ distinct points.
     Then, $\sC$
     corresponds to the hyperplane $\pi: m_1m_2X_1-(m_1+m_2)X_3-2am_1m_2X_4+X_6+a(m_1+m_2)X_7+a^2m_1m_2X_8=0$
     which passes through $Q$ if and only if
     \begin{equation}\label{eqsol}
     m_1m_2q_1- (m_1+m_2)q_3-2am_1m_2q_4+q_6+a(m_1+m_2)q_7=0.
   \end{equation}
   In particular, if we can determine $m_1,m_2$ and $a$ such that~\eqref{eqsol}
   is satisfied, then the point $Q$ is not special.

   A similar argument applies when $q_7=0$.
 \end{remark}

Let now $q \equiv 1$ $\mod 3$ and $\omega$ be a root of $T^2+T+1=0$
Consider a  non-singular plane cubic curve   $\sE$ over $\FF_q$ with  canonical equation: $$X^3+Y^3+Z^3-3cXYZ=0,$$ where $c \neq \infty, 1,\omega,\omega^2$.

If $c=1+\sqrt{3}$, then  the elliptic curve $\sE$ is  harmonic that is, $j(\sE) \neq 0$, see \cite[Lemma 11.47]{H}.
Using Remark~\ref{m:rem} and the symmetry $Y\leftrightarrow Z$ of the curve
$\sE$ it is possible to test for the completeness of $\nu_3^2(\sE )$.
 With the aid of GAP \cite{gap}, we see that for $q=121$ we obtain a curve with $n=144$ rational points,  for $q=157,169$ we obtain curves with $n=180$ rational points whereas for $q=179$ we get a curve with $n=180$ points and in each case the $n$ rational points define a complete NMDS code.

%
%


\section*{Acknowledgements}

This research was carried out within the activities of the GNSAGA - Gruppo Nazionale per le Strutture Algebriche, Geometriche e le loro Applicazioni of the Italian INdAM.

\end{document}